\documentclass[11pt]{article}
\usepackage[utf8]{inputenc}
\usepackage{authblk}
\usepackage{mathtools}
\usepackage[english]{babel}
\usepackage{amsmath}
\usepackage{amsthm}
\usepackage{amsfonts}
\usepackage{amssymb}
\usepackage{graphicx}
\usepackage[usenames,dvipsnames]{color}
\usepackage{hyperref}
\usepackage[left=2cm,right=2cm,top=2cm,bottom=2cm]{geometry}
\usepackage{dsfont}
\usepackage{pgf,tikz,pgfplots}
\usepackage{enumerate}

\usepackage[normalem]{ulem}
\usepackage{comment}

\theoremstyle{plain}
\newtheorem{theorem}{Theorem}[section]
\newtheorem{corollary}[theorem]{Corollary}

\newtheorem{lemma}[theorem]{Lemma}

\newtheorem{conjecture}[theorem]{Conjecture}

\theoremstyle{definition}

\newtheorem{remark}[theorem]{Remark}

\newtheorem{problem}[theorem]{Problem}

\newtheorem{def/prop}[theorem]{Definition/Proposition}

\pgfplotsset{compat = 1.16}

\DeclareMathOperator{\eps}{\varepsilon}
\DeclareMathOperator{\boldc}{\mathbf{c}}
\newcommand{\col}{\mathrm{Col}}

\title{\scshape Large monochromatic components in colorings of complete hypergraphs}

\author[1]{Lyuben Lichev}
\author[2]{Sammy Luo\thanks{Luo was supported by NSF GRFP Grant DGE-1656518}}
\affil[1]{Department of Mathematics, Univ. Jean Monnet, Saint-Etienne, France}
\affil[2]{Department of Mathematics, Stanford University, Stanford, CA 94305, USA}
\begin{document}

\maketitle

\begin{abstract}
Gyárfás famously showed that in every $r$-coloring of the edges of the complete graph $K_n$, there is a monochromatic connected component with at least $\frac{n}{r-1}$ vertices. A recent line of study by Conlon, Tyomkyn, and the second author addresses the analogous question about monochromatic connected components with many edges. 
In this paper, we study a generalization of these questions for $k$-uniform hypergraphs. Over a wide range of extensions of the definition of connectivity to higher uniformities, we provide both upper and lower bounds for the size of the largest monochromatic component that are tight up to a factor of $1+o(1)$ as the number of colors grows. We further generalize these questions to ask about counts of vertex $s$-sets contained within the edges of large monochromatic components. We conclude with more precise results in the particular case of two colors.
\end{abstract}

\section{Introduction}
Erd\H{o}s and Rado made the simple but beautiful observation that every $2$-coloring of the edges of $K_n$ contains a monochromatic spanning tree, or equivalently, a spanning monochromatic connected component. 
Gy\'arf\'as~\cite{Gya77} generalized this observation by showing that any edge-coloring of $K_n$ in $r\ge 2$ colors contains a monochromatic connected component with at least $\tfrac{n}{r-1}$ vertices. In the other direction, he gave a construction showing that this bound is tight when $r-1$ is a prime power and $n$ is a multiple of $(r-1)^2$. 
Since then, many further extensions have been considered, including ones in which $K_n$ is replaced with a graph of high minimum degree \cite{GS17}, a nearly-complete bipartite graph \cite{DKS20}, or a sparse random graph \cite{BD17,BB21}; the large monochromatic component can also be taken to have small diameter~\cite{Rus12}.

Recently, Conlon and Tyomkyn~\cite{CT21} initiated a closely related line of research aimed at studying the number of \emph{edges} in the monochromatic components in $r$-colorings of the edges of $K_n$. In the case $r=2$, they showed that every such coloring contains a monochromatic connected component with at least $\tfrac{2}{9}n^2+O(n^{3/2})$ edges, and this bound is asymptotically tight. By examining Gy\'arf\'as' construction, they showed that when $r-1$ is a prime power, there is an $r$-coloring of $K_n$ in which every monochromatic component contains at most $\tfrac{1}{r(r-1)} \tbinom{n}{2} + O(n)$ edges, and they conjectured that this upper bound is tight when $r=3$. Luo~\cite{Luo21} confirmed this conjecture by showing that every 3-coloring of $K_n$ contains a component with at least $\lceil \frac{1}{6}\tbinom{n}{2}\rceil$ edges. Moreover, by giving a lower bound for the largest number of edges in a connected component in a graph of a given density, he derived a lower bound of $\frac{1}{r^2-r+5/4}\tbinom{n}{2}$ for a general number of colors $r$. 
Later, Conlon, Luo, and Tyomkyn~\cite{CLT22} proved a tight lower bound of $\lceil \frac{1}{12}\tbinom{n}{2}\rceil$ in the case of four colors, and conjectured that a lower bound of $\frac{1}{r(r-1)}\tbinom{n}{2}$ holds for any number of colors $r$.

The above questions for both vertices and edges extend naturally to the hypergraph setting: Given an $r$-coloring of the edges of the complete $k$-uniform hypergraph $K_n^{k}$, what can we say about the largest number of vertices or edges in a monochromatic connected component? The problems here are significantly more complex than in the graph setting, and much less is known. In fact, to even make sense of the question, one needs to choose one of multiple equally natural generalizations of the definition of a connected component in a $k$-uniform hypergraph. In the setting of what we will later call \emph{$1$-tight} components, F\"uredi and Gy\'arf\'as \cite{FG91} showed that such an $r$-edge-coloring of $K_n^{k}$ always contains a component with at least $n/q$ vertices, where $q$ is the smallest integer satisfying $r\leq q^{k-1}+q^{k-2}+\cdots+q+1$, and moreover, this is sharp when $q^k|n$ and the affine space of dimension $k$ and order $q$ (consisting of $q^k$ points) exists. For more work on related problems about vertex counts in monochromatic connected components of a hypergraph, see \cite{BDDE19,DT20,GHM19,GH09}.

In the current paper, we analyze hypergraph versions of the problems studied by Gy\'arf\'as~\cite{Gya77} and Conlon, Luo, and Tyomkyn~\cite{CLT22,CT21,Luo21}. In fact, we significantly generalize the above setting by studying analogous questions about counts of vertex sets of size between $1$ and $k$ contained inside some edge of a monochromatic component over a range of different notions of connectivity. 
To our knowledge, our paper is the first to introduce a general framework for this hypergraph coloring problem.

\subsection{Notation and terminology} We start by introducing some useful notation and terminology that will enable us to state our main results. For an integer $k\ge 2$, we often refer to $k$-uniform hypergraphs simply as \emph{$k$-graphs}.
For integers $k$ and $n$ such that $2\le k\le n$, we denote by $K^k_n$ the complete $k$-graph on $n$ vertices, and we denote its vertex set by $V_n$. For a hypergraph $H$, we denote by $V(H)$ its vertex set and identify $H$ with its edge set. In particular, we write $|H|$ for the number of edges in $H$. Moreover, we denote by $\Delta(H)$ the maximum degree of $H$, and by $\Delta_2(H)$ the maximum co-degree of $H$ (that is, the maximum size of an edge set whose intersection contains at least two vertices). Furthermore, for a positive integer $s$, we define
\[E^{(s)}(H) = \{S\subseteq V(H): |S| = s, \exists e\in H:\: S\subseteq e\},\]
and call it the \emph{$s$-shadow of $H$}. 

The \emph{chromatic index of $H$}, denoted $\chi'(H)$, is the minimum number of colors needed to color the edges of $H$ so that every two edges with non-empty intersection have different colors. Moreover, for an integer $r\ge 2$, we define $\col_r(H)$ as the family of (not necessarily proper) edge-colorings of $H$ in $r$ colors (which we index by the integers $1,2,\ldots,r$ for convenience).
Colorings in $\col_r(H)$ are denoted by $\boldc$ (possibly with additional upper or lower indices) to avoid confusion with constants.

For a hypergraph $H$, the \emph{$t$-tight components of $H$} are constructed as follows. Starting from $|H|$ sets, each consisting of a single edge in $H$, we iteratively merge two sets if the first contains an edge $e$ and the second contains an edge $f$ such that $|e\cap f|\ge t$. When no further mergings are possible, the obtained family of edge sets, denoted $TC_t(H)$, is the set of $t$-tight components of $H$.

For a positive integer $n$, we let $[n] = \{1,\ldots,n\}$, and by an \emph{$n$-set} we mean a set of size $n$. For positive integers $\ell$ and $k$, denote by $(\ell)_k$ the product $\prod_{i=0}^{k-1} (\ell - i)$. For real numbers $t$, $a$, $b > 0$ and $c > 0$, we sometimes write $t=(a\pm b)c$ to mean that $t\in [(a-b)c, (a+b)c]$. Finally, for an integer $s$ and a set $S$, we denote by $\tbinom{S}{s}$ the family of $s$-subsets of $S$.

\subsection{Main results}

Fix positive integers $k$, $t$, $n$ and $r\ge 2$ satisfying $t \le k-1$. For a coloring $\boldc\in \col_r(K^k_n)$ and an integer $i\in [r]$, we denote by $H_i = H_i(\boldc)$ the $k$-graph of all edges with color $i$ in $\boldc$. We will be interested in the $t$-tight components of the $k$-graphs $(H_i)_{i=1}^r$. More precisely, for a positive integer $s\in [k]$ and a coloring $\boldc\in \col_r(K^k_n)$, we denote
\[M(n,r,k,t,s; \boldc) = \max_{i\in [r]} \max_{C\in TC_t(H_i)} |E^{(s)}(C)|.\]
In other words, $M(n,r,k,t,s; \boldc)$ is the largest number of vertex $s$-tuples that are contained in the edges of a monochromatic $t$-tight component in the coloring $\boldc$. Also, set
\[M(n,r,k,t,s) = \min_{\boldc\in \col_r(K^k_n)} M(n,r,k,t,s; \boldc).\]
For example, $(k,t,s)=(2,1,1)$ corresponds to Gy\'arf\'as's original problem, so we have $M(n,r,2,1,1)\ge \frac{n}{r-1}$, with equality when $r-1$ is a prime power and $(r-1)^2|n$. The case $(k,t,s)=(2,1,2)$ corresponds to $M(n,r)$ as defined in~\cite{CLT22} and~\cite{Luo21}, so $M(n,r,2,1,2)\ge \frac{1}{r^2-r+5/4}\binom{n}{2}$. Also, note that the result of F\"uredi and Gy\'arf\'as mentioned earlier \cite{FG91} implies $M(n,r,k,1,1)\ge \frac{n}{r^{1/(k-1)}}$.

Inspired by these cases where lower bounds are known, our first main result provides a general lower bound for $M(n,r,k,t,s)$.

\begin{theorem}\label{thm:LB}
For all positive integers $k$, $t$, $s$, $r$ and $n$ satisfying $\max(t+1,s)\le k$,
\[M(n,r,k,t,s) \ge r^{-s/(k-t)} \binom{n}{s}.\]
\end{theorem}
\noindent
Note that the condition $\max(t+1,s)\le k$ imposes no restrictions: it incorporates all values for which $M(n,r,k,t,s)$ is well-defined.

It is natural to ask how tight the lower bound given by Theorem~\ref{thm:LB} is. The next theorem provides a corresponding asymptotic upper bound, showing that the lower bound is tight to within a factor of $1+o(1)$ as $r$ and $n$ grow.

\begin{theorem}\label{thm:UB}
Fix a real number $\eps > 0$ and positive integers $k$, $t$ and $s$ satisfying $\max(t+1,s)\le k$. Then, there is a positive integer $r_{\eps}$ such that for every $r\ge r_{\eps}$, we have
\begin{equation}\label{eq:UB}
M(n,r,k,t,s) \le (1+\eps)r^{-s/(k-t)} \binom{n}{s}
\end{equation}
for all sufficiently large $n$.
\end{theorem}

As seen in~\cite{CLT22,CT21,Luo21}, obtaining precise results for a fixed number of colors can be a much more intricate matter that requires careful structural arguments and additional combinatorial ideas. In general, it is not even clear a priori that for fixed $r$, $k$, $t$ and $s$, $\binom{n}{s}^{-1} M(n,r,k,t,s)$ has a limit as $n\to \infty$. Our next theorem confirms that this limit exists.

\begin{theorem}\label{thm:Lambda}
For all positive integers $r$, $k$, $t$ and $s$ satisfying $\max(t+1,s)\le k$, there is a constant $\Lambda = \Lambda(r,k,t,s)$ such that 
\[\lim_{n\to \infty} \binom{n}{s}^{-1} M(n,r,k,t,s) = \Lambda.\]
\end{theorem}

We now narrow our focus to the case $r=2$. Our next result generalizes the observation that any $2$-edge-coloring of a complete graph contains a monochromatic spanning tree, yielding analogous structural results for certain ranges of $t$ and $s$.

\begin{theorem}\label{thm:r=2} 
Let $n$, $k$, $t$, and $s$ be positive integers such that $\max(t+1,s)\le k\le n$.
\begin{enumerate}
    \item[(a)] If $2\max(t,s)\le k$, then $M(n,2,k,t,s) = \tbinom{n}{s}$. That is, in every $2$-edge-coloring of $K^k_n$, there is a monochromatic $t$-tight component $C$ such that $E^{(s)}(C)=E^{(s)}(K_n^k)$.
    \item[(b)] If $3\max(t,s)\le 2k$, then there is a $2$-edge-coloring $\boldc$ of $K^k_n$ attaining 
\[M(n,2,k,t,s;\boldc) = M(n,2,k,t,s),\]
and such that each of the colors has a single $t$-tight component containing all edges of that color. 
\end{enumerate}
\end{theorem}

Extending the techniques used to show Theorem~\ref{thm:r=2}, we are able to provide the following more precise description of $M(n,2,3,t,s)$ for any appropriate choice of $t$ and $s$, thus mostly settling the question about 2-colorings of the complete 3-uniform hypergraph.

\begin{corollary}\label{cor:r=2,k=3}
For every $n\ge 3$,
\begin{align*}
& M(n,2,3,1,1) = n,\quad M(n,2,3,1,2) = \binom{n}{2} - \binom{\lfloor n/2\rfloor}{2},\\
& M(n,2,3,2,1) = n,\quad M(n,2,3,2,2) = \binom{n}{2} - \binom{\lfloor n/2\rfloor}{2}.
\end{align*}
Moreover, 
\[\lim_{n\to \infty} \binom{n}{3}^{-1} M(n,2,3,1,3) = 6\sqrt{21}-27\approx 0.495\]
and 
\[0.24 \le \lim_{n\to \infty} \binom{n}{3}^{-1} M(n,2,3,2,3)\le \frac{3}{8}.\]
\end{corollary}

\paragraph{Outline of the proofs.} The proof of Theorem~\ref{thm:LB} combines Lovasz's version of the Kruskal-Katona theorem (Lemma~\ref{lem:kruskal-katona}) and a careful averaging argument over the $t$-tight components of $G$, reminiscent of the arguments in~\cite{CLT22,Luo21}.

The proof of Theorem~\ref{thm:UB} is somewhat more involved. First, we show that~\eqref{eq:UB} holds for some $n=n_0$ by constructing an appropriate $r$-coloring of the complete $k$-graph on $n_0$ vertices. The correctness of the construction is justified by a famous result of Keevash on the existence of designs~\cite{Kee14} and a theorem by Pippenger and Spencer on the chromatic index of nearly regular $k$-graphs with small co-degree~\cite{PS89}. We then use a balanced blow-up of the above construction to obtain the bounds for larger $n$. Roughly speaking, given $n\ge n_0$, we partition the vertices of $V_n$ into $n_0$ approximately equal parts, with the vertices in each part viewed as copies of a fixed vertex in $V_{n_0}$. Then, the chosen coloring $\boldc$ on $K^k_{n_0}$ induces a coloring on the edges in $K^k_n$ with vertices in $k$ different parts. The most delicate part of the proof consists of designing an appropriate coloring of the edges with multiple vertices in the same part. (Note that while there will be relatively few such edges, they may serve to merge large $t$-tight components whose sizes we aim to reduce.) This is done by suitably extending the coloring of $K^k_{n_0}$ to all vertex sets of size between 1 and $k$, and defining a coloring of the edges in $K^k_n$ according to the subset of parts they intersect (thus extending $\boldc$). We believe that the same technique may be of independent interest and serve in other contexts as well. As a representative example, we apply the same technique to prove Theorem~\ref{thm:Lambda}.

The proofs of Theorem~\ref{thm:r=2} and Corollary~\ref{cor:r=2,k=3} are of a different flavor. Both use an elementary structural argument generalizing the simple observation that there is a monochromatic spanning tree in every 2-coloring of $K_n$; in Corollary~\ref{cor:r=2,k=3}, this is combined with a non-trivial optimization step, which in the case $(t,s)=(2,3)$ makes use of Theorem~\ref{thm:LB}. The second part of Theorem~\ref{thm:r=2} relies on a more detailed structural argument. We show that there is a red $t$-tight component $C_r$ and a blue $t$-tight component $C_b$ such that the following holds for properly chosen $\ell$ and $m$ such that $\ell+m\le k$: for every $\ell$-set $L\subseteq V_n$, the family of $(\ell+m)$-sets containing $L$ is contained either in $E^{(\ell+m)}(C_r)$ or in $E^{(\ell+m)}(C_b)$. Then, by analyzing what happens when we recolor edges outside $C_r \cup C_b$, we are able to transform any optimal coloring to one where $C_r\cup C_b = K^k_n$. 

\paragraph{Plan of the paper.} This paper is organized as follows. In Section~\ref{sec:2}, we prove Theorem~\ref{thm:LB}, our general lower bound. In Section~\ref{sec:3}, we study general upper bounds, developing tools that culminate in the proof of Theorem~\ref{thm:UB}. We also apply these tools to prove Theorem~\ref{thm:Lambda}. In Section~\ref{sec:4}, we discuss strategies for obtaining more precise results for $r=2$, proving Theorem~\ref{thm:r=2} and Corollary~\ref{cor:r=2,k=3}. We finish the paper with some concluding remarks and open questions in Section~\ref{sec:conclusion}.

\section{Lower Bounds}\label{sec:2}

We start by recalling the following well-known and convenient form of the Kruskal-Katona Theorem due to Lov\'asz~\cite{Lov93}.

\begin{lemma}[\cite{Lov93}, Exercise 31 (b)]\label{lem:kruskal-katona}
Let $G$ be a $k$-uniform hypergraph with $|G|=\binom{x}{k}$ for some real number $x\geq k$. Then, for every positive integer $s\le k$, we have
\[
|E^{(s)}(G)|\geq \binom{x}{s}.
\]
\end{lemma}

The key result in showing the lower bound is the following.

\begin{lemma}\label{lem:lb-density-general}
Fix positive integers $k$ and $t$ satisfying $t\le k-1$, and a real number $\delta\in [0,1]$. Then, in every $k$-uniform hypergraph $G$ with $|G|\geq \delta\binom{n}{k}$, there is a $t$-tight connected component $C$ such that for every positive integer $s\leq k$, 
\[
|E^{(s)}(C)|\geq \delta^{\frac{s}{k-t}}\binom{n}{s}.
\]
\end{lemma}
\begin{proof}
By definition the $t$-tight components of $G$ induce a partition of $E^{(t)}(G)$. Therefore, 
\[
    \frac{|G|}{|E^{(t)}(G)|}=\frac{\sum_{C\in TC_t(G)}|C|}{\sum_{C\in TC_t(G)}|E^{(t)}(C)|},
\]
so a standard averaging argument shows that there is some component $C\in TC_t(G)$ such that

\begin{equation}\label{eqn:kk}
\frac{|C|}{|E^{(t)}(C)|}\geq \frac{|G|}{|E^{(t)}(G)|}\geq \frac{\delta\binom{n}{k}}{\binom{n}{t}}.
\end{equation}
Let $x\in [k,n]$ be a real number such that $|C|=\binom{x}{k}$. By Lemma~\ref{lem:kruskal-katona} we then have $|E^{(s)}(C)|\geq \binom{x}{s}$ for all $s\in [k]$. It suffices then to show that $\binom{x}{s}\geq \delta^{\frac{s}{k-t}}\binom{n}{s}$ for every $s\leq k$. Indeed, by~\eqref{eqn:kk} we have
\begin{align*}
\delta \leq \frac{|C|}{|E^{(t)}(C)|} \frac{\binom{n}{t}}{\binom{n}{k}} \leq \frac{\binom{x}{k}}{\binom{x}{t}} \frac{\binom{n}{t}}{\binom{n}{k}} = \prod_{i=t}^{k-1} \frac{x-i}{n-i}.
\end{align*}
Since $k\leq x\leq n$, the quantity $\frac{x-i}{n-i}$ is nonincreasing in $i$ for $i\in [0,k-1]$. Thus, for every $j\leq t$, we have
\[
\frac{x-j}{n-j}\geq \left(\prod_{i=t}^{k-1} \frac{x-i}{n-i}\right)^{\frac{1}{k-t}}\geq \delta^{\frac{1}{k-t}},
\]
and likewise, for every $j\in [t,k-1]$, we have 
\[
\prod_{i=t}^{j} \frac{x-i}{n-i} \geq \left(\prod_{i=t}^{k-1} \frac{x-i}{n-i}\right)^{\frac{j-t+1}{k-t}}\geq \delta^{\frac{j-t+1}{k-t}}.
\]
Hence, for every $s\leq k$, we have
\begin{align*}
\frac{\binom{x}{s}}{\binom{n}{s}} &=\prod_{i=0}^{s-1}\frac{x-i}{n-i} =  \left(\prod_{i=0}^{\min(t,s)-1}\frac{x-i}{n-i}\right)\left(\prod_{i=\min(t,s)}^{s-1}\frac{x-i}{n-i}\right)\geq \delta^{\frac{\min(t,s)}{k-t}} \delta^{\frac{\max(0,s-t)}{k-t}}=\delta^{\frac{s}{k-t}},
\end{align*}
and so $|E^{(s)}(C)|\geq \delta^{\frac{s}{k-t}}\binom{n}{s}$, as desired.
\end{proof}

Theorem~\ref{thm:LB} follows from Lemma~\ref{lem:lb-density-general} by an immediate density argument.

\begin{proof}[Proof of Theorem~\ref{thm:LB}]
Fix any $r$-coloring of the edges of $K_n^k$. Then, there is some color $i\in [r]$ given to at least $\frac{1}{r}\binom{n}{k}$ edges. Applying Theorem~\ref{lem:lb-density-general} to $G = H_i$ yields a monochromatic $t$-tight component $C$ such that for all $s\leq k$, $|E^{(s)}(C)|\geq r^{-\frac{s}{k-t}}\binom{n}{s}$. This shows that $M(n,r,k,t,s)\geq r^{-\frac{s}{k-t}}\binom{n}{s}$, as desired.
\end{proof}

\section{Upper Bounds and Convergence Results}
\label{sec:3}

In this section, we prove Theorems~\ref{thm:UB} and~\ref{thm:Lambda}. We begin with the proof of Theorem~\ref{thm:UB}, which is divided into two parts. First, we give a construction for a single value of $n$ based on a coloring of a Steiner system. Then, we use a blow-up strategy to generalize the construction to all sufficiently large $n$.

\subsection{\texorpdfstring{Upper bounds for fixed $n$}{}}
We begin by reviewing some facts about Steiner systems and hypergraph colorings which will be used in our proof of Theorem~\ref{thm:UB}. Recall that an \emph{$(n, h, k)$-Steiner system} is a family $F$ of $h$-subsets of an $n$-set $X$ such that every $k$-subset of $X$ is included in exactly one set in $F$. In this paper, we view $(n, h, k)$-Steiner systems as $h$-graphs. The following theorem due to Keevash~\cite{Kee14} shows that for all sufficiently large $n$, the trivial divisibility assumptions are the only obstruction for the existence of an $(n, h, k)$-Steiner system.

\begin{theorem}[see Theorems 1.4 and 1.10 in~\cite{Kee14}]\label{thm:Keevash}
For all pairs of positive integers $k,h$ such that $h > k$ there is $n_0 = n_0(k, h)$ such that the following holds: for every $n\ge n_0$ satisfying $\tbinom{h-i}{k-i} \mid \tbinom{n-i}{k-i}$ for all $i\in [0,k-1]$, there exists an $(n,h,k)$-Steiner system.
\end{theorem}

Another tool in our analysis is the following theorem of Pippenger and Spencer~\cite{PS89} on proper edge-colorings of hypergraphs.

\begin{theorem}[see Theorem 1.1 in~\cite{PS89}]\label{thm:PS}
For every $k\ge 2$ and $\eps > 0$ there exist $\delta=\delta(k,\eps) > 0$ and a positive integer $n_0=n_0(k,\eps)$ such that the following holds: for any $k$-graph $H$ with at least $n_0$ vertices, minimum degree at least $(1-\delta)\Delta(H)$, and maximum codegree $\Delta_2(H)\le \delta \Delta(H)$, we have $\chi'(H)\le (1+\eps) \Delta(H)$.
\end{theorem} 

The following is the key lemma in our proof of Theorem~\ref{thm:UB} for a single value of $n$.

\begin{lemma}\label{lem:gen Gyarfas}
Fix positive integers $k$ and $t$ satisfying $k > t$ and a real number $\eps\in (0,1)$. Then, for every $h\ge k$ and sufficiently large $r\ge r_{0}=r_{0}(\eps,h,k,t)$, there is a positive integer $n = n(r,h,k,t) > h$ such that: 
\begin{enumerate}[(i)]
\item $r = (1\pm \eps) \prod_{i=t}^{k-1} \frac{n-i}{h-i}$;
\item there is an $(n, h, k)$-Steiner system whose $h$-sets may be partitioned into at most $r$ $h$-graphs so that every two edges in the same $h$-graph intersect in at most $t-1$ vertices.
\end{enumerate}
\end{lemma}
We remark that the case $t=1$ corresponds to the affine space construction used in showing the upper bound in \cite{FG91}.
\begin{proof}
Fix $h\ge k$, and let $\delta = \delta(\binom{h}{t}, \tfrac{\eps}{4})$, $n_0=n_0(\binom{h}{t}, \tfrac{\eps}{4})$ as defined in Theorem~\ref{thm:PS}. Let $n = h + \alpha (h)_k$ for some positive integer $\alpha$. Then, the fact that $h - i$ divides $n-i$ for all $i\in [0, k-1]$ ensures that there is an integer $\alpha'_0$ such that the assumptions of Theorem~\ref{thm:Keevash} are satisfied for all $\alpha\ge \alpha_0'$, and in this case an $(n,h,k)$-Steiner system $F$ exists. We recall that every $t$-subset of $V(F)$ is contained in exactly $\tfrac{(n-t)_{k-t}}{(h-t)_{k-t}}$ $h$-sets in $F$. Note that
\[\frac{(n-t)_{k-t}}{(h-t)_{k-t}} = \prod_{i=t}^{k-1} \left(1 + \alpha \frac{(h)_k}{h-i}\right),\]
and let us denote by $P(\alpha)$ the above polynomial function of $\alpha$ of degree $k-t$. Then, we define $\alpha_0$ as the smallest integer that satisfies the following conditions:
\begin{enumerate}[(i)]
    \item $\alpha_0\ge \alpha_0'$;
    \item for every $\alpha\ge \alpha_0$, $\tfrac{P(\alpha+1)}{P(\alpha)} \le 1 + \frac{\eps}{8}$;
    \item $\tbinom{h+\alpha_0 (h)_k}{t}\ge n_0$;
    \item $\frac{h-t}{n-t}=\frac{h-t}{h+\alpha(h)_k-t} \le \delta$ for $\alpha \ge \alpha_0$.
\end{enumerate}
Observing that $(1+\tfrac{\eps}{4})(1+\tfrac{\eps}{8}) < 1+\tfrac{\eps}{2}$ and using property (ii) from the definition of $\alpha_0$, the intervals $[(1+\tfrac{\eps}{4})P(\alpha), (1+\tfrac{\eps}{2})P(\alpha)]_{\alpha\ge \alpha_0}$ cover the interval $[(1+\tfrac{\eps}{4})P(\alpha_0), \infty)$. In particular, for every integer $r\ge r_{0} := \lceil(1+\tfrac{\eps}{4})P(\alpha_0)\rceil$ one may find $\alpha \ge \alpha_0$ for which
$r\in [(1+\tfrac{\eps}{4})P(\alpha), (1+\tfrac{\eps}{2})P(\alpha)]$. Fix any such $r$ and $\alpha$.

Now, it remains to apply Theorem~\ref{thm:PS} to the auxiliary hypergraph $H$ with vertex set $\tbinom{V(F)}{t}$ and edge set $\{\tbinom{f}{t}: f\in F\}$. Note that this hypergraph is $\binom{h}{t}$-uniform and regular of degree $P(\alpha)$ and maximum codegree $\tfrac{(n-t-1)_{k-t-1}}{(h-t-1)_{k-t-1}} = \tfrac{h-t}{n-t} P(\alpha)$, and by (iii), $|V(H)|= \binom{n}{t}\ge n_0$. Now, since $\tfrac{h-t}{n-t}\le \delta$ by (iv), we get that $\Delta_2(H)\le \delta \Delta(H)$. Therefore, since $r\ge (1+\tfrac{\eps}{4})\Delta(H)$, we conclude by Theorem~\ref{thm:PS} that $H$ may be properly edge-colored in $r$ colors. This yields the desired construction, as such an $r$-coloring of $H$ corresponds to a partition of the edges of $F$ where no two edges in the same part share $t$ or more vertices.
\end{proof}

We can now obtain~\eqref{eq:UB} for a single value of $n$ as a corollary. In fact, we prove a slightly stronger statement, which will be useful in the next part of the proof.

\begin{corollary}\label{cor:gen Gyarfas}
For every $\eps\in (0,1)$, positive integers $t$ and $k > t$, and sufficiently large $r\ge r' = r'(\eps,k,t)$, there is a positive integer $n_0\ge 3\eps^{-1}k^4 r^{1/(k-t)}$ and a coloring $\boldc_0\in \col_r(K_{n_0}^k)$ such that for all $s\le k$,
\[M(n_0,r,k,t,s;\boldc_0) \le (1+\eps) r^{-s/(k-t)} \binom{n_0}{s}.\]
\end{corollary}
\begin{proof}
Fix $h=\lceil 8\eps^{-1}k^4\rceil$, $r'=r_0(\frac{\eps}{8k},h,k,t)$ as in Lemma~\ref{lem:gen Gyarfas}, and $r\ge r'$. Then, for some $n_0 > h$, there is an $(n_0,h,k)$-Steiner system $F$ with vertex set $[n_0]$ given by applying Lemma~\ref{lem:gen Gyarfas} with parameter $\frac{\eps}{8k}$ instead of $\eps$. We denote by $H_1, \ldots, H_r$ the partition of $F$ into $r$ $h$-graphs provided by part~(ii) of the lemma. 

Now, we construct an $r$-coloring $\boldc_0$ of $K_{n_0}^k$ as follows: for each $S\in \binom{[n_0]}{k}$, give it color $i$ if $S\in E^{(k)}(H_i)$. Note that by construction, this choice of color is unique. Moreover, each monochromatic $t$-tight component in this coloring is isomorphic to $K^k_h$, and thus contains $\tbinom{h}{s}$ $s$-sets for each $s\le k$. Thus, combining the fact that for every $i\in [0,k-1]$,
\[\frac{h-k}{n_0-k}\le \frac{h-i}{n_0-i} = \frac{1+\tfrac{k-i}{h-k}}{1+\tfrac{k-i}{n_0-k}} \cdot \frac{h-k}{n_0-k} \le \left(1 + \frac{k}{h-k}\right) \frac{h-k}{n_0-k}\le \left(1 + \frac{\eps}{8k}\right) \frac{h-k}{n_0-k},\]
with part~(i) of Lemma~\ref{lem:gen Gyarfas} yields
\begin{align*}
\binom{h}{s} = \left(\prod_{i=0}^{s-1} \frac{h-i}{n_0-i}\right) \binom{n_0}{s}
&\le \left(1+\frac{\eps}{8k}\right)^{s}\left(\frac{h-k}{n_0-k}\right)^{s} \binom{n_0}{s} \le \left(1+\frac{\eps}{8k}\right)^{s}\left(\prod_{i=t}^{k-1} \frac{n_0-i}{h-i}\right)^{
-s/(k-t)} \binom{n_0}{s}\\
&\le \left(1+\frac{\eps}{8k}\right)^{s}\frac{(1+\frac{\eps}{8k})^{s/(k-t)}}{r^{s/(k-t)}} \binom{n_0}{s} \le \frac{e^{\eps/2}}{r^{s/(k-t)}} \binom{n_0}{s} \le (1+\eps)r^{-s/(k-t)} \binom{n_0}{s},
\end{align*}
where for the last two inequalities we used that $e^{x/2}\le 1+x\le e^x$ for every $x\in (0,1)$. Finally, crude bounds yield
\[
\frac{(h-s)^s}{s!}\le (1+\eps)r^{-s/(k-t)}\frac{n_0^s}{s!},
\]
or equivalently
\[
n_0 \ge (1+\eps)^{-1/s}(h-s)r^{1/(k-t)}\ge 3\eps^{-1}k^4 r^{1/(k-t)},
\]
which finishes the proof.
\end{proof}

\subsection{\texorpdfstring{Upper bounds for large $n$}{}}

Our aim now is to extend the construction from the previous section in order to prove \eqref{eq:UB} for all sufficiently large $n$. Our technique for doing so is captured in the following lemma, which gives a recursive upper bound on our quantities of interest via a careful blow-up construction.
\begin{lemma}\label{lem:blowup}
Let $n$, $n_0$, $r$, $k$, $t$, and $s$ be positive integers such that $n\ge n_0\ge k\ge \max(s,t+1)$. Then, for any coloring $\boldc_0\in \col_r(K_{n_0}^k)$, there is a coloring $\boldc\in \col_r(K_n^k)$ such that
\[
M(n,r,k,t,s;\boldc)\le \sum_{\ell=1}^s \left\lceil \frac{n}{n_0-k+1}\right\rceil^s \binom{s-1}{\ell-1} M(n_0,r,k,t,\ell;\boldc_0).
\]
\end{lemma}
\begin{proof}
Denote $N = n_0 - k + 1$ and set $\{v_{N+1}, \ldots, v_{n_0}\} = V_{n_0}\setminus V_N$. To begin with, we define an auxiliary coloring $\boldc_0'$ of the subsets of $V_N$ of size between 1 and $k$. More precisely, for every $L\subseteq V_N$ with $1\le|L|\le k$, we define $\varphi(L)=L\cup \{v_{N+1}, \ldots, v_{N+k-|L|}\}$ and $\boldc_0'(L) = c_0(\varphi(L))$.

Now, fix $n\ge n_0$. We partition the vertices of $K^k_n$ into $N$ parts $\{A_v\}_{v\in V_N}$ indexed by the vertices in $V_N$ and with sizes that are as equal as possible. For simplicity, assume that $n=mN$ for some positive integer $m$, so $|A_v|=m$ for all $v\in V_N$; the general case presents no further difficulty. For every $W\subseteq V_n$, we define $I_W = \{v\in V_N: W\cap A_v\neq \emptyset\}$ (in particular, for every edge $e\in K^k_n$, $I_e = \{v\in V_N: e\cap A_v\neq \emptyset\}$). We construct a coloring $\boldc$ of $K^k_n$ by setting $\boldc(e) = \boldc_0'(I_e)$ for every $e\in K^k_n$.

We will show that for every $e,f\in K^k_n$, $|e\cap f|\le |\varphi(I_e)\cap \varphi(I_f)|$. First, note that $|e\cap f| \le |I_e\cap I_f| + \min(k-|I_e|, k-|I_f|)$. Indeed, if $I_e\cap I_f = \{u_1, \ldots, u_\ell\}$ for some $\ell\le k$, and $a_i = |e\cap A_{u_i}|$ and $b_i = |f\cap A_{u_i}|$, then 
\[|e\cap f|\le \min\left(\sum_{i=1}^\ell a_i, \sum_{i=1}^\ell b_i\right) = \ell + \min\left(\sum_{i=1}^\ell (a_i-1), \sum_{i=1}^{\ell} (b_i-1)\right)\le |I_e\cap I_f| + \min(k - |I_e|, k - |I_f|).\]
However, $\{v_{N+1}, \ldots, v_{N+k-|I_e|}\}\subseteq \varphi(I_e)$ and a similar inclusion holds for $I_f$, so
\[|\varphi(I_e)\cap \varphi(I_f)|\ge |I_e\cap I_f| + \min(k-|I_e|, k-|I_f|).\]
Combining the above two statements shows that $|e\cap f|\le |\varphi(I_e)\cap \varphi(I_f)|$ for all $e,f\in K^k_n$, as claimed. In particular, every monochromatic $t$-tight component $C$ of $\boldc$ is contained in the preimage of a monochromatic $t$-tight component $C_0$ of $\boldc_0$ under the map $e\mapsto \varphi(I_e)$.

Now, fixing $\ell\in [s]$, we wish to count the number of $s$-sets $S\in E^{(s)}(C)$ with $|I_S|=\ell$. Each such $S$ is contained in an edge $e\in C$ with $\varphi(I_e)\in C_0$. Since $I_S\subseteq I_e\subseteq \varphi(I_e)$, we have $I_S\in E^{(\ell)}(C_0)$. On the other hand, for a given $L\in E^{(\ell)}(C_0)$, the number of $S\in E^{(s)}(C)$ with $I_S=L$ is at most $m^s \binom{s-1}{\ell-1}$: indeed,
there are $\binom{s-1}{\ell-1}$ sequences $(s_v)_{v\in L}$ of positive integers that sum up to $s$, and once we fix $s_v:=|S\cap A_v|$ for all $v\in L$, there are at most $m^s$ choices for $S$. Summing over all $\ell$ and $L\in E^{(\ell)}(C_0)$ then gives
\[
|E^{(s)}(C)|\le \sum_{\ell=1}^s m^s \binom{s-1}{\ell-1} |E^{(\ell)}(C_0)|\le \sum_{\ell=1}^s m^s \binom{s-1}{\ell-1} M(n_0,r,k,t,\ell;\boldc_0).
\]
Taking the maximum over all components $C$ then gives the desired bound.
\end{proof}

With this lemma, we can now easily complete the proof of Theorem~\ref{thm:UB}.

\begin{proof}[Proof of Theorem~\ref{thm:UB}]
Fix $\eps$, $k$, $t$, and $s$ as in the theorem statement. Let $r_{\eps}=r'(\frac{\eps}{4},k,t)$. By Corollary~\ref{cor:gen Gyarfas} (applied with $\tfrac{\eps}{4}$ instead of $\eps$) there is some $n_0\ge 12\eps^{-1} k^4 r^{1/(k-t)}$ and some coloring $\boldc_0\in \col_{r}(K_{n_0}^r)$ such that for all $\ell\le k$,
\[
M(n_0,r,k,t,\ell;\boldc_0)\le \left(1+\frac{\eps}{4}\right)r^{-\ell/(k-t)}\binom{n_0}{\ell}.
\]
Let $N=n_0-k+1$, and again assume for simplicity that $n=mN$ for some positive integer $m$. Then, Lemma~\ref{lem:blowup} yields a coloring $\boldc\in \col_r(K_n^k)$ such that
\[
M(n,r,k,t,s;\boldc)\le \sum_{\ell=1}^s m^s \binom{s-1}{\ell-1}M(n_0,r,k,t,\ell;\boldc_0)\le \sum_{\ell=1}^s m^s \binom{s-1}{\ell-1}\left(1+\frac{\eps}{4}\right)r^{-\ell/(k-t)}\binom{n_0}{\ell}.
\]
For $\ell\le s-1$, we have 
\begin{align*}
    \binom{s-1}{\ell-1} r^{-\ell/(k-t)}\binom{n_0}{\ell} &\le s^{s-\ell}r^{-s/(k-t)}r^{(s-\ell)/(k-t)}\binom{n_0}{s}\frac{\binom{n_0}{\ell}}{\binom{n_0}{s}} = r^{-s/(k-t)}\binom{n_0}{s} (s r^{1/(k-t)})^{s-\ell} \prod_{i=\ell}^{s-1} \frac{s+\ell-i}{n_0-i}\\
    &\le  r^{-s/(k-t)} \binom{n_0}{s} \left(s r^{1/(k-t)} \frac{s}{n_0-s} \right)^{s-\ell}
    .
\end{align*}
By our choice of $n_0$, we have 
\[
s r^{1/(k-t)} \frac{s}{n_0-s} \le \frac{k^2 r^{1/(k-t)}}{11 \eps^{-1} k^4 r^{1/(k-t)} } \le \frac{\eps}{4 k^2},
\]
so that
\[
\binom{s-1}{\ell-1} r^{-\ell/(k-t)}\binom{n_0}{\ell} \le  r^{-s/(k-t)} \binom{n_0}{s}  \left(\frac{\eps}{4 k^2}\right)^{s-\ell} \le \frac{\eps}{4k^{2}} r^{-s/(k-t)} \binom{n_0}{s}.
\]
Therefore,
\[
M(n,r,k,t,s;\boldc)\le r^{-s/(k-t)} m^s \binom{n_0}{s} \left(1+\frac{\eps}{4}\right)\left(1+ (s-1) \frac{\eps}{4k^{2}} \right) \le \left(1+\frac{\eps}{3}\right)r^{-s/(k-t)}m^s\binom{n_0}{s}.
\]
Moreover, we have that
\[
\binom{n_0}{s} \le \binom{N}{s} \left(\frac{n_0}{N-s}\right)^s \le \left(1+\frac{\eps}{4k^{3}}\right)^s \binom{N}{s},
\]
which leads to
\[
M(n,r,k,t,s;\boldc) \le \left(1+\frac{\eps}{3}\right) \left(1+\frac{\eps}{4k^{3}}\right)^s r^{-s/(k-t)}m^s\binom{N}{s} \le e^{\eps/2}r^{-s/(k-t)} \binom{n}{s} \le (1+\eps)r^{-s/(k-t)} \binom{n}{s},
\]
where we have again used the fact that $e^{\eps/2}\le 1+\eps$.

Thus, we may conclude that
\[
M(n,r,k,t,s)\le M(n,r,k,t,s; \boldc) \le (1+\eps)r^{-s/(k-t)}\binom{n}{s}
\]
for every $s\in [k]$, as desired.
\end{proof}

\subsection{\texorpdfstring{Convergence of $\tbinom{n}{s}^{-1} M(n,r,k,t,s)$}{}}
Now, we turn to the proof of Theorem~\ref{thm:Lambda}. By applying Lemma~\ref{lem:blowup} again, we will show that $\tbinom{n}{s}^{-1} M(n,r,k,t,s)$ converges to its inferior limit for all appropriate choices of $r$, $k$, $t$, and $s$.

\begin{proof}[Proof of Theorem~\ref{thm:Lambda}]
Fix $r$, $k$, $t$ and $s$ as in the statement of the theorem, and fix $\eps>0$. Define 
\[\Lambda = \liminf_{n\to \infty} \binom{n}{s}^{-1} M(n,r,k,t,s).\]
Note that Theorem~\ref{thm:LB} ensures that $\Lambda > 0$. By definition, there exist arbitrarily large values of $n_0$ such that $M(n_0,r,k,t,s)\le (\Lambda+\tfrac{\eps}{8}) \tbinom{n_0}{s}$. Fix some sufficiently large $n_0$, and let $\boldc_0$ be an optimal coloring of $K^k_{n_0}$. 

As before, let $N=n_0-k+1$ and $n=mN$ for some positive integer $m$ (the case of general $n$, as usual, presents no additional difficulty). By Lemma~\ref{lem:blowup}, there is some coloring $\boldc\in \col_r(K_n^k)$ such that
\[
M(n,r,k,t,s;\boldc)\le \sum_{\ell=1}^s m^s \binom{s-1}{\ell-1}M(n_0,r,k,t,\ell;\boldc_0) \le m^s \left(\Lambda+\frac{\eps}{8}\right) \binom{n_0}{s} + m^s \sum_{\ell=1}^{s-1} \binom{s-1}{\ell-1} \binom{n_0}{\ell}.
\]
Moreover, for every $\ell\in [s-1]$, we can crudely bound $\binom{s-1}{\ell-1} \binom{n_0}{\ell}$ from above by $\frac{1}{n_0} s^{2s} \binom{n_0}{s}$. Thus, choosing $n_0\ge 8\eps^{-1} k s^{2s}$ yields
\[
M(n,r,k,t,s;\boldc) \le m^s \binom{n_0}{s}\left(\Lambda+\frac{\eps}{8} + \frac{\eps}{8}\right).
\]
This choice of $n_0$ also gives
\[
\binom{n_0}{s}\le \binom{N}{s}\left(\frac{n_0}{N-s}\right)^s \le \left(1+\frac{\eps}{4s}\right)^s \binom{N}{s},
\]
and thus
\[
M(n,r,k,t,s)\le M(n,r,k,t,s;\boldc) \le (\Lambda+\tfrac{\eps}{4}) \left(1+\tfrac{\eps}{4s}\right)^s m^s \binom{N}{s} \le (\Lambda+\tfrac{\eps}{4})(1+\tfrac{\eps}{2}) \binom{n}{s}\le (\Lambda+\eps)\binom{n}{s},
\]
where we have again used the fact that $1+\tfrac{x}{2}\le e^{x/2} \le 1+x$ for $x\in (0,1)$. Since $\eps>0$ can be chosen arbitrarily small, this shows that $\tbinom{n}{s}^{-1} M(n,r,k,t,s)$ converges to $\Lambda$, as desired.
\end{proof}

\section{\texorpdfstring{The case $r=2$}{}}
\label{sec:4}

In this section, we fix $r=2$ and call the two colors red and blue.

\begin{proof}[Proof of Theorem~\ref{thm:r=2}]
Fix a 2-coloring $\boldc\in \col_2(K_n^k)$. First, suppose that $\max(t,s)\le \frac{k}{2}$. Denote $m = \lfloor \tfrac{k}{2}\rfloor$, $\alpha = \tbinom{n}{m}$, and set 
\[\{U_1, \ldots, U_{\alpha}\} = \binom{V_n}{m}.\]
Consider the complete graph $K$ with vertex set $(U_i)_{i=1}^{\alpha}$, and construct an auxiliary coloring on $K$ as follows: for every edge $U_iU_j$ in $K$, color it red (respectively blue) if there is a red (respectively blue) $k$-edge in $\boldc$ that contains $U_i\cup U_j$, choosing one of the two colors arbitrarily if both are valid. Then, by the observation at the beginning of the introduction, this $2$-coloring yields a monochromatic spanning tree $T$. Moreover, every two incident edges of $T$ correspond to $k$-edges intersecting in at least $m\ge t$ vertices. Thus, there is a monochromatic $t$-tight component covering all $m$-subsets of $V_n$, and since $s\le m$, all $s$-subsets of $V_n$ as well. This proves part~(a) of the theorem. Part~(b) follows in this case by considering the coloring where every edge is red.

It remains to prove part~(b) of the theorem when $\frac{k}{2}< \max(t,s)\le \frac{2k}{3}$. In this case, define $\ell = \lfloor \tfrac{k}{3}\rfloor$ and $m = \lfloor \frac{k-\ell}{2}\rfloor$. 
Fix an $\ell$-set $L\subset V_n$ and consider the coloring $\boldc_L$ on the complete $(k-\ell)$-graph on $V_n\setminus L$ that assigns to each $(k-\ell)$-subset $S$ of $V_n\setminus L$ the color of $S\cup L$ in $\boldc$. By an immediate case analysis on the residue of $k$ modulo $3$, we can verify that $t\le \lfloor \frac{2k}{3}\rfloor = \ell+m$. Hence, by part~(a) of the theorem, there is a monochromatic $(t-\ell)$-tight component in $\boldc_L$ whose edges contain all $m$-subsets of $V_n\setminus L$, and hence a monochromatic $t$-tight component $C_L$ in $\boldc$ such that $E^{(\ell+m)}(C_L)$ contains all $(\ell+m)$-subsets of $V_n$ containing $L$.

Now, associate to each $\ell$-set $L$ the color of $C_L$ (choosing arbitrarily if needed). For every edge $e\supset L$ of the same color as $L$ (say, red), we claim that $e\in C_L$. Indeed, for any $(\ell+m)$-subset $S$ of $e$ containing $L$, by definition we have $S\subset e'$ for some $e'\in C_L$. But $|e\cap e'|\ge |S|=\ell+m\ge t$, so $e$ and $e'$ are in the same red component, as claimed. This implies that for any two $\ell$-sets $L_1$ and $L_2$ of the same color (say, red), $C_{L_1}$ and $C_{L_2}$ coincide since $2\ell \le \ell+m$ and thus there is some red edge containing $L_1\cup L_2$, which must then be in both $C_{L_1}$ and $C_{L_2}$. Let $C_r$ and $C_b$ be the $t$-tight components associated to all red and blue $\ell$-sets, respectively. By definition, every $(\ell+m)$-set is then contained in an edge in at least one of $C_r$ and $C_b$, that is, $E^{(\ell+m)}(K_n^k)=E^{(\ell+m)}(C_r)\cup E^{(\ell+m)}(C_b)$. Since $s\le \lfloor \frac{2k}{3}\rfloor = \ell+m$, we have that $E^{(s)}(K_n^k)=E^{(s)}(C_r)\cup E^{(s)}(C_b)$.

Finally, suppose that there is a red edge $e$ outside $C_r$. We show that recoloring it blue adds it to $C_b$ but does not increase $E^{(s)}(C_b)$ (of course, the same holds when the colors are reversed). Indeed, by the observations above, all $\ell$-sets contained in $e$ must be blue, so every $(\ell+m)$-subset of $e$ (and hence every $s$-subset) is contained in an edge of $C_b$. Moreover, since $t\le \ell+m$, recoloring $e$ in blue adds it to $C_b$. Hence, recoloring all red edges outside $C_r$ in blue and all blue edges outside $C_b$ in red concludes the proof of the theorem.
\end{proof}

\begin{proof}[Proof of Corollary~\ref{cor:r=2,k=3}]
Theorem~\ref{thm:r=2}~(a) immediately implies that $M(n,2,3,1,1)=n$ for all $n\ge 3$. In addition, for any $2$-edge-coloring $\boldc$, our arguments from the proof of Theorem~\ref{thm:r=2}~(b) with $\ell = m = 1$ yield a partition $V_n=V_r \cup V_b$ (one of which may be empty) such that for some 2-tight red component $C_r$, all vertex pairs intersecting $V_r$ are in $E^{(2)}(C_r)$, and the analogous statement holds for a blue component $C_b$. Hence, at least one of $C_r$ and $C_b$ spans all vertices of $K^3_n$, while 
\[\max(E^{(2)}(C_r), E^{(2)}(C_b)) \ge \binom{n}{2} - \binom{\lfloor n/2\rfloor}{2}.\]
Equality is attained in a coloring of $K^3_n$ in which the sets $V_r$ and $V_b$ have sizes as equal as possible, and an edge $e$ is colored red if $|e\cap V_r| > |e\cap V_b|$ and blue otherwise. In fact, in this coloring, $C_r$ and $C_b$ are also monochromatic $1$-tight components whose union contains all edges of $K_n^3$. Hence, we have shown that $M(n,2,3,2,1)=n$ and $M(n,2,3,1,2) = M(n,2,3,2,2) = \binom{n}{2} - \tbinom{\lfloor n/2\rfloor}{2}$ for all $n\ge 3$.

Now, consider the case $(t,s) = (1,3)$. From the arguments above, there is a monochromatic $2$-tight component, and hence a monochromatic $1$-tight component, that spans $V(K_n^3)$. Without loss of generality, assume that this spanning $1$-tight component is blue, and call it $C_b'$. If there is also a spanning red $1$-tight component, then one of the two contains at least $\left\lceil\tfrac{1}{2}\tbinom{n}{3}\right\rceil$ edges. Otherwise, let the vertex sets of the red components have sizes $j_1\ge \ldots \ge j_\ell$ for some $\ell\ge 2$. First, suppose that $j_1\le \lceil \tfrac{n}{2}\rceil$. Then, since $x\mapsto \tbinom{x}{3}$ is a convex function on $[3, +\infty)$, we may deduce that $\sum_{i=1}^{\ell} \tbinom{j_i}{3}$ is maximized when $\ell = 2$, $j_1 = \lceil \tfrac{n}{2}\rceil$, and $j_2 =  \lfloor \tfrac{n}{2}\rfloor$. Then, the number of edges in the blue component is at least 
\[\binom{n}{3} - \sum_{i=1}^{\ell} \binom{j_i}{3} \ge \binom{n}{3} - \binom{\lfloor n/2\rfloor}{3} - \binom{\lceil n/2\rceil}{3}\ge \frac{1}{2}\binom{n}{3}.\]
Now, suppose that $j_1\ge \lceil \tfrac{n}{2}\rceil+1$. Note that $|C_b'|\ge \tbinom{n}{3} - \tbinom{j_1}{3} - \tbinom{n-j_1}{3}$. On the other hand, $C_b'$ and the red component on $j_1$ vertices together contain at least $\binom{n}{3}-\tbinom{n-j_1}{3}$ edges, so 
\[M(n,2,3,1,3)\ge \min_{j_1\ge \lceil\frac{n}{2}\rceil+1} \max\left( \tbinom{n}{3} - \tbinom{j_1}{3} - \tbinom{n-j_1}{3}, \tfrac{1}{2}\left(\tbinom{n}{3}-\tbinom{n-j_1}{3}\right)\right).\] 
The first term in the maximum is decreasing for $j_1\ge \frac{n}{2}$ while the second term is increasing, and the two are equal when $2\tbinom{j_1}{3} + \tbinom{n-j_1}{3} = \tbinom{n}{3}$. Thus, the quantity is minimized when $x:=\frac{j_1}{n}$ satisfies $2x^3+(1-x)^3=1+o(1)$. The unique root of the equation $2x^3+(1-x)^3 = 1$ between $\tfrac{1}{2}$ and 1 is $x_0 = \tfrac{\sqrt{21}-3}{2}$, so we have
\[M(n,2,3,1,3) \ge (1-x_0^3-(1-x_0)^3+o(1))\binom{n}{3} = (6\sqrt{21}-27+o(1))\binom{n}{3}.\]
This lower bound is attained by the 2-coloring consisting of two red $3$-cliques of sizes $\lfloor x_0 n\rfloor$ and $n - \lfloor x_0 n\rfloor$, with all remaining edges colored in blue.

Finally, consider the case $(t,s)=(2,3)$. For the upper bound, take a partition $(U_r, U_b)$ of $V_n$ into two parts that are as equal as possible. Color an edge red if it intersects $U_r$ in $1$ or $3$ vertices, and blue otherwise. One may easily check that in this coloring, there are $4$ monochromatic $2$-tight components, and the largest one has $\tbinom{\lceil n/2\rceil}{2}\lfloor \tfrac{n}{2}\rfloor$ edges.

For the lower bound, fix a 2-coloring $\boldc$ of $K_n^3$. As before, let $V_r$, $V_b$, $C_r$, $C_b$ be such that every red edge intersecting $V_r$ is in $C_r$, and similarly for blue. Suppose that $|V_r|\ge |V_b|$, and denote by $y$ the fraction of edges contained in $V_r$ that are blue. Now, on the one hand, applying Lemma~\ref{lem:lb-density-general} to the blue edges contained in $V_r$ yields a blue component $C_1$ of size $|C_1|\ge y^3 \tbinom{|V_r|}{3}$. On the other hand, all red edges with at least one vertex in $V_r$ are contained in $C_r$, and analogously for $C_b$. In particular, $|C_r|\ge (1-y)\binom{|V_r|}{3}$, and
\[\max(|C_r|,|C_b|)\ge \frac{1}{2}\left(\binom{n}{3} - \binom{|V_b|}{3} - y\binom{|V_r|}{3}\right).\]
Combining these bounds yields
\[M(n,2,3,2,3;\boldc)\ge \max\left(y^3 \binom{|V_r|}{3}, (1-y)\binom{|V_r|}{3}, \frac{1}{2}\left(\binom{n}{3} - \binom{|V_b|}{3} - y\binom{|V_r|}{3}\right)\right).\]
Setting $x = \tfrac{|V_r|}{n}\in [0.5, 1]$, this shows that
\[M(n,2,3,2,3)\ge \left(\min_{x\in [0.5,1],\, y\in [0,1]} \max\left(y^3x^3, (1-y)x^3, \frac{1-(1-x)^3-yx^3}{2}\right) + o(1)\right ) \binom{n}{3}\ge 0.24 \binom{n}{3},\]
where the last inequality holds for all sufficiently large $n$.\footnote{See \scalebox{0.58}{\url{https://www.wolframalpha.com/input?i=find+minimum+of+max\%28y\%5E3x\%5E3\%2C\%281-y\%29x\%5E3\%2C\%281-\%281-x\%29\%5E3-y*x\%5E3\%29\%2F2\%29+for+0.5+\%3C+x+\%3C\%3D+1+and+0+\%3C+y+\%3C\%3D+1}}.}
\end{proof}
\begin{remark}
\label{rem:2323}
In the case $(t,s)=(2,3)$, it is possible to derive stronger lower bounds by applying an averaging argument directly to the ratio $\frac{|C|}{|E^{(2)}(C)|}$ for each blue component $C$. A computer optimization on the resulting system of inequalities suggests a lower bound of $M(n,2,3,2,3)\ge z\binom{n}{3}$, where $z\approx 0.318$ is the real root of the polynomial $(1-z)^3-z=0$. Directly verifying this lower bound appears to be a routine but heavily computational matter; we therefore omit the details for the sake of simplicity.
\end{remark}

\section{Concluding remarks}\label{sec:conclusion}
In this paper, we analyzed in detail the properties of large monochromatic $t$-tight components in $r$-edge-colorings of the complete $k$-graph $K^k_n$. Nevertheless, much remains unknown.

One central question concerns the bounds on $M(n,r,k,t,s)$ for a fixed number of colors $r$. Theorem~\ref{thm:Lambda} shows the existence of the limit $\Lambda=\Lambda(r,k,t,s)=\lim_{n\to \infty} \binom{n}{s}^{-1}M(n,r,k,t,s)$. While Theorem~\ref{thm:LB} and Theorem~\ref{thm:UB} together provide asymptotically matching bounds on $M(n,r,k,t,s)$, and therefore on $\Lambda(r,k,t,s)$, they do not give tight bounds for a fixed number of colors. For example, in the graph case $(k,t,s)=(2,1,2)$, Theorem~\ref{thm:LB} gives a lower bound of $M(n,r,2,1,2)\ge \frac{1}{r^2}\binom{n}{2}$ that is identical to the ``trivial'' lower bound on $M(n,r)$ from Corollary~1.4 in~\cite{Luo21}, which is weaker than the best known general lower bound of $\frac{1}{r^2-r+5/4}\binom{n}{2}$. As in that case, and as seen from the arguments and results in Section~\ref{sec:4}, intricate structural arguments will be needed to obtain more precise bounds for fixed $r$. We believe that establishing such improved bounds for general $r$ (and ideally, determining the constant $\Lambda(r,k,t,s)$) is an interesting and ambitious open problem.

\begin{problem}
For any fixed positive integers $r,k,t,s$ satisfying $\max(t+1,s)\le k$, determine the value of the limit
\[
\Lambda(r,k,t,s)=\lim_{n\to\infty} \binom{n}{s}^{-1}M(n,r,k,t,s).
\]
\end{problem}

Theorem~\ref{thm:r=2} suggests a promising approach towards obtaining better bounds on $\Lambda$ in the simple setting when $r=2$ and $\max(t,s)\le \frac{2k}{3}$. In this case, we can make the strong structural assumption that all edges in an optimal coloring are contained in two $t$-tight components, one in each color.

Finally, there is more to say about the case $(r,k,t,s) = (2,3,2,3)$, the smallest unsolved case outside of the graph setting. We conjecture that the upper bound in  Corollary~\ref{cor:r=2,k=3} is correct in this case.
\begin{conjecture}\label{conj:5.2}
$\Lambda(2,3,2,3)=\frac{3}{8}$.
\end{conjecture}
The heuristic justification for this conjecture is as follows: As mentioned in Remark~\ref{rem:2323}, a more careful argument can seemingly improve the lower bound on $\Lambda(2,3,2,3)$ to at least $z\approx 0.318$. This value arises from the configuration where a single $2$-tight component contains all edges in one of the two colors (say, red). Indeed, if the red monochromatic component has at most $z\binom{n}{3}$ edges, Lemma~\ref{lem:lb-density-general} then shows that there is a blue monochromatic component of size at least $(1-z)^3\tbinom{n}{3} = z\tbinom{n}{3}$. We suspect, however, that a much stronger lower bound than given by Lemma~\ref{lem:lb-density-general} might hold in this case, and perhaps in general for large densities. For example, it is conjectured in \cite[Section~11]{LangSanhueza} that a $3$-graph with edge density at least $\frac{5}{8}+o(1)$ contains a $2$-tight component with at least $(\frac{1}{2}+o(1))\binom{n}{3}$ edges. 
This would immediately imply a lower bound of $(\frac{3}{8}+o(1)) \binom{n}{3}$ in the configuration described above.

\paragraph{Note.} In a personal communication, Richard Lang informed us that his upcoming work with Mathias Schacht and Jan Volec solves our Conjecture~\ref{conj:5.2} using flag algebraic techniques.

\bibliography{Refs}
\bibliographystyle{plain}
\end{document}